\documentclass[10pt]{amsart}

\usepackage{amsfonts}
\usepackage{amsthm}
\usepackage{amsmath}
\usepackage{graphicx}
\usepackage{amscd,amssymb,amsthm}

\newcounter{minutes}
\divide\time by 60
\newcounter{hours}
\multiply\time by 60 \addtocounter{minutes}{-\time}

\newcommand{\ki}{\operatorname{Ki}}
\newcommand{\dt}{{\rm d}t}
\newcommand{\ds}{{\rm d}s}
\newcommand{\dy}{{\rm d}y}
\newcommand{\f}{\phi}

\title{Functional inequalities for the Bickley function}

\author{\'Arp\'ad Baricz}
\address{Department of Economics, Babe\c{s}-Bolyai University,
Cluj-Napoca 400591, Romania} \email{bariczocsi@yahoo.com}

\author{Tibor K. Pog\'any}
\address{Faculty of Maritime Studies, University of Rijeka, Rijeka 51000, Croatia}
\email{poganj@brod.pfri.hr}



\newtheorem{theorem}{Theorem}

\begin{document}

\def\thefootnote{}
\footnotetext{ \texttt{File:~\jobname .tex,
          printed: \number\year-0\number\month-0\number\day,
          \thehours.\ifnum\theminutes<10{0}\fi\theminutes}
} \makeatletter\def\thefootnote{\@arabic\c@footnote}\makeatother

\maketitle
\allowdisplaybreaks

\begin{abstract}
In this paper our aim is to deduce some complete monotonicity properties and functional inequalities for the Bickley function. The key tools in our proofs are the classical integral inequalities, like Chebyshev, H\"older-Rogers, Cauchy-Schwarz, Carlson and Gr\"uss inequalities, as well as the monotone form of l'Hospital's rule. Moreover, we prove the complete monotonicity of a determinant function of which entries involve the Bickley  function.
\end{abstract}

\section{\bf Introduction}
\setcounter{equation}{0}

The Bessel function fractional integral
   $$\ki_{\alpha}(x)=\frac{1}{\Gamma(\alpha)}\int_x^{\infty}(t-x)^{\alpha-1}K_0(t)\dt,$$
where $K_0$ is the modified Bessel function of the second of zero order, was first introduced for $\alpha\in\{1,2,\dots\}$ by Bickley \cite{bickley} in connection with the solution of heat convection problems. This function appears also in neutron transport calculations,
and is frequently used in nuclear reactor computer codes. An alternative representation of the Bickley function, which will be used frequently in the sequel, is the following
   \begin{equation} \label{A0}
      \ki_{\alpha}(x) = \int_0^{\infty}e^{-x\cosh t}(\cosh t)^{-\alpha}\dt,
   \end{equation}
where $\alpha$ is an arbitrary real number and $x>0.$ For properties of the Bickley function, including asymptotic
expansions and generalizations we refer to \cite{amos1,amos2,blair,hanna}, \cite[Chapter 8]{luke}, \cite{milgram}, \cite[p. 259]{nist} and to the references therein.

In this paper, by using the classical integral inequalities, like
Chebyshev, H\"older-Rogers, Cauchy-Schwarz, Carlson and Gr\"uss, and the monotone form of l'Hospital's rule
we present some complete monotonicity properties and functional inequalities for the Bickley function. Moreover, we prove the complete monotonicity of a determinant function of which entries involve the Bickley function. For similar functional inequalities involving other special functions we refer for example to the papers \cite{kratzel,baricz} and to the references therein.

Before we present the main results of this paper we recall some definitions, which will be used in the sequel. A function
$f \colon (0,\infty)\rightarrow\mathbb{R}$ is said to be completely monotonic if $f$ has derivatives of all orders and satisfies
   $$(-1)^mf^{(m)}(x)\geq 0$$
for all $x>0$ and $m\in\{0,1,2,\dots\}.$

The exponentially convex functions form a sub-class of convex functions introduced by Bernstein in \cite{bernstein} (see also \cite{AJPR}). A function $g \colon \mathcal I \mapsto \mathbb R$ is exponentially convex on $\mathcal I \subseteq \mathbb R$ if it is continuous and
   \[ H_\xi(f) = \sum_{j, k = 1}^n \xi_j\xi_k f(x_j+x_k) \ge 0\, ,\]
for all $n\in\{1,2,\dots\}$ and all $\xi_j \in \mathbb R,$ $j\in\{1,2,\dots,n\}$ such that  $x_j + x_k \in \mathcal I$ for $j,k\in\{0,1,\dots,n\}.$

A function $h \colon (0,\infty)\to(0,\infty)$ is said to be logarithmically convex, or simply log-convex, if its natural logarithm $\ln h$ is convex, that is, for all $x,y>0$ and $\lambda\in[0,1]$ we have
   $$h(\lambda x+(1-\lambda)y) \leq \left[h(x)\right]^{\lambda}\left[h(y)\right]^{1-\lambda}.$$
A similar characterization of log-concave functions also holds. We also note that every completely monotonic function is log-convex, see
\cite[p. 167]{widder}. The same conclusion holds true for the exponentially convex functions on $(0,\infty)$, that is,
if $h:(0,\infty)\to(0,\infty)$ is exponentially convex, then it is log-convex. See \cite[Corollary 2]{AJPR} for more details.

By definition, a function $q \colon (0,\infty)\rightarrow(0,\infty)$ is said to be geometrically (or multiplicatively)
convex if it is convex with respect to the geometric mean, that is, if for all $x,y>0$ and all $\lambda\in[0,1]$ the inequality
   $$q(x^{\lambda}y^{1-\lambda}) \leq[q(x)]^{\lambda}[q(y)]^{1-\lambda}$$
holds. The function $q$ is called geometrically concave if the above inequality is reversed. Observe that, actually the
geometrical convexity of a function $q$ means that the function $\ln q$ is a convex function of $\ln x$ in
the usual sense. We also note that the differentiable function $h$ is log-convex (log-concave) if and only if
$x \mapsto h'(x)/h(x)$ is increasing (decreasing), while the differentiable function $q$ is geometrically convex (concave) if
and only if the function $x \mapsto xq'(x)/q(x)$ is increasing (decreasing). See for example \cite{geom} for more details.

Finally, let us recall the concept of relative convexity. This concept has been
considered by Hardy et al. \cite[p. 75]{hardy}: if
$\varphi,\psi:[a,b]\to \mathbb{R}$ are two continuous functions and
$\psi$ is strictly monotone, then we say that $\varphi$ is convex
(concave) with respect to $\psi$ if $\varphi\circ\psi^{-1}$ is convex (concave) in the
usual sense on the interval $\psi([a,b]).$ The usual convexity of a
function $\varphi$ in this manner means actually that the function $\varphi$ is
convex with respect to the identity function, the log-convexity of
$\varphi$ is exactly the fact that the function $\ln \varphi$ is convex with
respect to the identity function, while the geometrical convexity of
$\varphi$ means that $\ln \varphi$ is convex with respect to logarithm function.
See \cite{niculescu} for more details. It is also known (see
\cite{cargo}) that the increasing function $\varphi$ is convex with
respect to an increasing function $\psi$ if and only if the function
$\varphi'/\psi'$ is increasing, or if and only if the inequality $\psi''(x)/\psi'(x)\leq \varphi''(x)/\varphi'(x)$ is valid
for all $x\in (a,b).$

\section{\bf Bickley function: Monotonicity patterns and functional inequalities}
\setcounter{equation}{0}

Our first main result is the following theorem.

\begin{theorem}\label{th1}
The following assertions are true:
\begin{enumerate}
\item[\bf a.] The function $x\mapsto \ki_{\alpha}(x)$ is completely
monotonic on $(0,\infty)$ for all $\alpha\in\mathbb{R}.$
\item[\bf b.] The function $\alpha\mapsto \ki_{\alpha}(x)$ is completely
monotonic on $\mathbb{R}$ for all $x>0.$
\item[\bf c.] The function $\alpha\mapsto \ki_{\alpha}(x)$ is log-convex on $\mathbb{R}$ for all $x>0.$
\item[\bf d.] The function $x\mapsto \ki_{\alpha}(x)$ is log-convex on $(0,\infty)$ for all $\alpha \in \mathbb{R}.$
\item[\bf e.] The function $x\mapsto \ki_{\alpha}(x)$ is geometrically concave on $(0,\infty)$ for all $\alpha\in\{-1,0,1,\dots\}.$ Consequently, for all $\alpha\in\{-1,0,1,\dots\}$ and $x,y>0$ we have
    \begin{equation}\label{ineq5}
       \ki_{\alpha}(\sqrt{xy})\geq\sqrt{\ki_{\alpha}(x)\ki_{\alpha}(y)}\geq\ki_{\alpha}\left(\frac{x+y}{2}\right),
    \end{equation}
    \begin{equation}\label{tura}
    1\leq \frac{\ki_{\alpha}(x)\ki_{\alpha-2}(x)}{\left[\ki_{\alpha-1}(x)\right]^2}\leq 1+\frac{\ki_{\alpha}(x)}{x\ki_{\alpha-1}(x)}.
    \end{equation}
Moreover, the right-hand side of \eqref{ineq5} and the left-hand side of \eqref{tura} are valid for all real $\alpha.$
\item[\bf f.] The inequality
    \begin{equation}\label{ineq1}
       \ki_{-\beta}(x)\ki_{\alpha+\beta}(x)\leq \ki_0(x)\ki_{\alpha}(x)
    \end{equation}
is valid for all $x>0$ and $\alpha+\beta\leq0\leq\beta$ or $\alpha+\beta\geq0\geq\beta.$ If $\alpha\geq0,$ $\alpha+\beta\geq0$ or $\beta\leq0,$ $\alpha+\beta\leq0,$ then \eqref{ineq1} is reversed. In particular, when $\beta=-1$ and $\alpha$ is changed to $\alpha-1$ the inequality \eqref{ineq1} becomes
\begin{equation}\label{relat}\frac{\ki_{\alpha}''(x)}{\ki_{\alpha}'(x)}=-\frac{\ki_{\alpha-2}(x)}{\ki_{\alpha-1}(x)}\geq
-\frac{\ki_{0}(x)}{\ki_{1}(x)}=\frac{\ki_{2}''(x)}{\ki_{2}'(x)},\end{equation}
where $\alpha\geq2,$ i.e. $x\mapsto -\ki_{\alpha}(x)$ is convex with respect to $x\mapsto -\ki_{2}(x)$ for $\alpha\geq 2.$
\item[\bf g.] If $\alpha+\beta\geq0,$ $\beta\leq0$ and $x>0,$ then the following inequality is valid \begin{equation}\label{ineq2} \left|\ki_0(x)\ki_{\alpha}(x)-\ki_{-\beta}(x)\ki_{\alpha+\beta}(x)\right|\leq\frac{\left[\ki_{0}(x)\right]^2}{4}.\end{equation}
\item[\bf h.] The inequality
    \begin{equation}\label{ineq3}
       \frac{2\Gamma\left(\frac{\alpha+1}{2}\right)}{\sqrt{\pi}\,\Gamma\left(\frac{\alpha}{2}\right)}
       \ki_{\alpha}(x)\ki_{\alpha}(y)\leq\ki_{\alpha}(x+y)\leq\ki_{\alpha}(x)+\ki_{\alpha}(y) \leq
       \ki_\alpha(x+y)+ \frac{\sqrt{\pi}\,\Gamma\left(\frac{\alpha}{2}\right)}{2\Gamma\left(\frac{\alpha +1}{2}\right)}\, .
    \end{equation}
Moreover, if we let $r, s \ge 1,$ then for all $\alpha, x, y>0,$
    \begin{equation} \label{ineq31}
       r\ki_\alpha(x) + s\ki_\alpha(y) \leq \ki_\alpha(rx+sy) + (r+s-1)
            \frac{\sqrt{\pi}\,\Gamma\left(\frac{\alpha}{2}\right)}{2\Gamma\left(\frac{\alpha +1}{2}\right)}\, .
    \end{equation}
\item[\bf i.] The inequality
    \begin{equation}\label{ineq4}
       \frac{1}{\ki_0(x)}\ki_{\alpha}(x)\ki_{\beta}(x)\leq\ki_{\alpha+\beta}(x)\leq\ki_{\alpha}(x)+\ki_{\beta}(x)\leq \ki_0(x)+\ki_{\alpha+\beta}(x)
    \end{equation}
holds for all $\alpha,\beta>0$ and $x>0.$
\item[\bf j.] The function $x\mapsto \ki_{\alpha}(x)$ is exponentially convex on $(0,\infty)$ for all $\alpha>0.$
\item[\bf k.] The function $\alpha\mapsto \ki_{\alpha}(x)$ is exponentially convex on $\mathbb{R}$ for all $x>0.$
\item[\bf l.] For all $\alpha,\beta \in \mathbb R$ and $x>0,$
    \begin{equation} \label{ineqX}
       \ki_{\alpha+\beta}(x) + \ki_{\alpha-\beta}(x) \ge 2\ki_\alpha(x)\, .
    \end{equation}
\item[\bf m.] The inequality
    \begin{equation} \label{ineqXX}
       \ki_{\alpha+\nu}(x)\ki_{\alpha-\mu}(x) + \ki_{\alpha-\nu}(x)\ki_{\alpha+\mu}(x)
            \ge 2\left[\ki_\alpha(x)\right]^2
    \end{equation}
holds for all $\alpha,\nu, \mu \in \mathbb R$ and $x>0$.
\item[\bf n.] The function $(\alpha,x)\mapsto \ki_{\alpha}(x)$ is log-convex for all $x>0$ and $\alpha\in\mathbb{R}.$ In particular,
\begin{equation}\label{logc2}
\left[\ki_{\alpha}(x)\right]^2\leq \ki_{\alpha(1+\mu)}((1+\nu)x)\ki_{\alpha(1-\mu)}((1-\nu)x)
\end{equation}
is valid for all $\alpha,\nu,\mu\in\mathbb{R}$ and $x>0.$
\end{enumerate}
\end{theorem}

\begin{proof}[\bf Proof]
{\bf a.} \& {\bf b.} It is known \cite[Theorem 4]{samko} that if the kernel $K(x,t)$ is completely monotonic in $x$ for all $t>0$ and $f$ is a nonnegative
locally integrable function such that the integral
$$\int_a^b\frac{\partial^n}{\partial x^n}K(x,t)f(t)\dt$$
converges uniformly for all $n\in\{0,1,2,\dots\}$ and $0\leq a<b\leq \infty$ in a neighborhood of any point $x>0,$ then the function
$$x\mapsto\int_a^bK(x,t)f(t)\dt$$ is completely monotonic on $(0,\infty).$ Now, since the function $x\mapsto e^{-x\cosh t}$ is completely monotonic on $(0,\infty)$ for all $t>0,$ the above result implies that indeed the function $x\mapsto \ki_{\alpha}(x)$ is completely
monotonic on $(0,\infty)$ for all $\alpha\in\mathbb{R}.$ Similarly, since the function $\alpha \mapsto (\cosh t)^{-\alpha}$ is completely monotonic on $\mathbb{R}$ for all $t>0,$ by using \cite[Theorem 4]{samko} again we obtain that the function $\alpha\mapsto \ki_{\alpha}(x)$ is completely
monotonic on $\mathbb{R}$ for all $x>0.$ It should be mentioned here that \begin{equation}\label{der}\ki'_{\alpha}(x)=-\ki_{\alpha-1}(x)\end{equation} and by induction we have $$(-1)^m\ki_{\alpha}^{(m)}(x)=\ki_{\alpha-m}(x)>0$$ for all $x>0,$ $\alpha\in\mathbb{R}$ and $m\in\{0,1,2,\dots\},$ which provides an alternative proof for part {\bf a}. Similarly,
$$(-1)^m\frac{\partial^m\ki_{\alpha}(x)}{\partial \alpha^m}=\int_0^{\infty}e^{-x\cosh t}(\cosh t)^{-\alpha}\left[\log(\cosh t)\right]^m\dt>0$$ for all $x>0,$ $\alpha\in\mathbb{R}$ and $m\in\{0,1,2,\dots\},$ which provides an alternative proof for part {\bf b}.

{\bf c.} \& {\bf d.} These results follow from parts {\bf a} \& {\bf b}, since every completely monotonic function is
log-convex (see \cite[p. 167]{widder}). However, we give here an alternative proof by using the classical H\"older-Rogers inequality for
integrals \cite[p. 54]{mitri},
   \begin{equation} \label{holder}
      \int_a^b|f(t)g(t)|\dt \leq {\left[\int_a^b|f(t)|^p\dt\right]}^{1/p}
           {\left[\int_a^b|g(t)|^q\dt\right]}^{1/q},
   \end{equation}
where $p>1,$ $1/p+1/q=1,$ $f$ and $g$ are real functions defined on $[a,b]$ and
$|f|^p,$ $|g|^q$ are integrable functions on $[a,b].$ Using \eqref{holder} we obtain that
\begin{align*}
\ki_{\lambda\alpha+(1-\lambda)\beta}(x)&=\int_0^{\infty}e^{-x\cosh t}(\cosh t)^{-(\lambda\alpha+(1-\lambda)\beta)}\dt\\
&=\int_0^{\infty}\left[e^{-x\cosh t}(\cosh t)^{-\alpha}\right]^{\lambda}\left[e^{-x\cosh t}(\cosh t)^{-\beta}\right]^{1-\lambda}\dt\\
&\leq\left[\int_0^{\infty}e^{-x\cosh t}(\cosh t)^{-\alpha}\dt\right]^{\lambda}\left[\int_0^{\infty}e^{-x\cosh t}(\cosh t)^{-\beta}\dt\right]^{1-\lambda}\\
&=\left[\ki_{\alpha}(x)\right]^{\lambda}\left[\ki_{\beta}(x)\right]^{1-\lambda}
\end{align*}
holds for all $\lambda\in[0,1],$ $\alpha,\beta\in\mathbb{R}$ and
$x>0,$ i.e. the function $\alpha\mapsto \ki_{\alpha}(x)$ is
log-convex on $\mathbb{R}.$ Similarly, by using \eqref{holder} we get
\begin{align*}
\ki_{\alpha}(\mu x+(1-\mu)y)&=\int_0^{\infty}e^{-(\mu x+(1-\mu)y)\cosh t}(\cosh t)^{-\alpha}\dt\\
&=\int_0^{\infty}\left[e^{-x\cosh t}(\cosh t)^{-\alpha}\right]^{\mu}\left[e^{-y\cosh t}(\cosh t)^{-\alpha}\right]^{1-\mu}\dt\\
&\leq \left[\int_0^{\infty}e^{-x\cosh t}(\cosh t)^{-\alpha}\dt\right]^{\mu}\left[\int_0^{\infty}e^{-y\cosh t}(\cosh t)^{-\alpha}\dt\right]^{1-\mu}\\
&=\left[\ki_{\alpha}(x)\right]^{\mu}\left[\ki_{\alpha}(y)\right]^{1-\mu}
\end{align*}
holds for all $\mu\in[0,1],$ $\alpha\in\mathbb{R}$ and
$x,y>0,$ i.e. the function $x\mapsto \ki_{\alpha}(x)$ is
log-convex on $(0,\infty).$

Alternatively, to prove part {\bf d} we may use part {\bf c} of this theorem. More
precisely, since the function $\alpha\mapsto \ki_{\alpha}(x)$ is log-convex, the following
Tur\'an type inequality holds for all
$\alpha_1,\alpha_2\in\mathbb{R}$ and $x>0$
\begin{equation}\label{turan}\left[\ki_{\frac{\alpha_1+\alpha_2}{2}}(x)\right]^2\leq
\ki_{\alpha_1}(x)\ki_{\alpha_2}(x).\end{equation} Now, if we
choose $\alpha_1=\alpha-2$ and $\alpha_2=\alpha$ and apply \eqref{der}, then we obtain
\begin{equation}\label{turanka}\left[\frac{\ki_{\alpha}'(x)}{\ki_{\alpha}(x)}\right]'=
\frac{\ki_{\alpha-2}(x)\ki_{\alpha}(x)-\left[\ki_{\alpha-1}(x)\right]^2}{\left[\ki_{\alpha}(x)\right]^2}\geq0,\end{equation}
i.e. the function $x\mapsto {\ki_{\alpha}'(x)}/{\ki_{\alpha}(x)}$ is increasing on
$(0,\infty)$ for all $\alpha\in\mathbb{R}.$

{\bf e.} To prove the asserted result, first we verify the following statement: For each real $\alpha$ if the function $\ki_{\alpha-1}$ is geometrically concave on $(0,\infty),$ then the function $\ki_{\alpha}$ is also geometrically concave on $(0,\infty).$ Since $\ki_{\alpha-1}$ is geometrically concave it follows that the function $$x\mapsto 1+\frac{x\ki_{\alpha-1}'(x)}{\ki_{\alpha-1}(x)}=\frac{\ki_{\alpha-1}(x)-x\ki_{\alpha-2}(x)}{\ki_{\alpha-1}(x)}=
-\frac{\left[x\ki_{\alpha-1}(x)\right]'}{\ki_{\alpha}'(x)}$$
is decreasing on $(0,\infty)$ and by the monotone form of l'Hospital's rule \cite[Lemma 2.2]{vuorinen} the function
$$x\mapsto -\frac{x\ki_{\alpha-1}(x)}{\ki_{\alpha}(x)}=\frac{x\ki_{\alpha}'(x)}{\ki_{\alpha}(x)}$$
is also decreasing on $(0,\infty),$ that is, the function $\ki_{\alpha}$ is geometrically concave on $(0,\infty).$ Here we used tacitly
that $x\ki_{\alpha-1}(x)$ and $\ki_{\alpha}(x)$ tend to zero as $x\to \infty.$ Now, because $\ki_0=K_0$ and $\ki_{-1}=-K_0'=K_1$ and
according to \cite[Theorem 2]{baricz} the function $K_{\alpha}$ is geometrically concave on $(0,\infty)$ for all $\alpha\in\mathbb{R},$ we
obtain that $\ki_{-1},$ $\ki_0,$ $\ki_1,$ $\ki_2,$ $\dots$ are geometrically concave on $(0,\infty).$

Now, we focus on the inequalities \eqref{ineq5} and \eqref{tura}. Inequality \eqref{ineq5} follows by definition. The left-hand
side of \eqref{tura} is a particular case of the Tur\'an type inequality \eqref{turan}, while the right-hand side of \eqref{tura} follows from the geometric concavity. More precisely, since $\ki_{\alpha}$ is geometrically concave, it follows that
$$\left[\frac{x\ki_{\alpha}'(x)}{\ki_{\alpha}(x)}\right]'=\left[\frac{\ki_{\alpha-1}(x)}{\ki_{\alpha}(x)}\right]^2
\left[-x-\frac{\ki_{\alpha}(x)}{\ki_{\alpha-1}(x)}+x\frac{\ki_{\alpha}(x)\ki_{\alpha-2}(x)}{\left[\ki_{\alpha-1}(x)\right]^2}\right]\leq0$$
for all $\alpha\in\{-1,0,1,\dots\}$ and $x>0.$

{\bf f.} We recall the Chebyshev integral inequality \cite[p.
40]{mitri}: If $f,g:[a,b]\rightarrow\mathbb{R}$ are integrable
functions, both increasing or both decreasing, and
$p:[a,b]\rightarrow\mathbb{R}$ is a positive integrable function,
then
\begin{equation}\label{csebisev}
\int_a^bp(t)f(t)\dt\int_a^bp(t)g(t)\dt\leq
\int_a^bp(t)\dt\int_a^bp(t)f(t)g(t)\dt.
\end{equation}
Note that if one of the functions $f$ or $g$ is decreasing and the
other is increasing, then \eqref{csebisev} is reversed. We shall use
this inequality. For this we write
$\ki_{\alpha}(x)$ as
$$\ki_{\alpha}(x)=\int_0^{\infty}e^{-x\cosh t}(\cosh t)^{\beta}(\cosh t)^{-(\alpha+\beta)}\dt$$
and let $p(t)=e^{-x\cosh t},$ $f(t)=(\cosh t)^{\beta}$ and $g(t)=(\cosh t)^{-(\alpha+\beta)}.$
The function $f$ is increasing (decreasing) on $(0,\infty)$ if and only if $\beta\geq0$ ($\beta\leq0$), while $g$ is increasing (decreasing) on $(0,\infty)$ if and only
if $\alpha+\beta\leq0$ ($\alpha+\beta\geq0$). Observe that
$$\int_0^{\infty}p(t)\dt=\int_0^{\infty}e^{-x\cosh t}\dt=\ki_{0}(x)=K_0(x),$$
$$\int_0^{\infty}p(t)f(t)\dt=\int_0^{\infty}e^{-x\cosh t}(\cosh t)^{\beta}\dt=\ki_{-\beta}(x)$$
and
$$\int_0^{\infty}p(t)g(t)\dt=\int_0^{\infty}e^{-x\cosh t}(\cosh t)^{-(\alpha+\beta)}\dt=\ki_{\alpha+\beta}(x).$$
Thus, appealing to Chebyshev integral inequality \eqref{csebisev}, the
proof of the inequality \eqref{ineq1} is complete.

Finally, if we consider the functions $\varphi,\psi:(0,\infty)\to\mathbb{R},$ defined by $\varphi(x)=-\ki_{\alpha}(x)$ and $\psi(x)=-\ki_2(x),$ then
by using the inequality \eqref{relat} we obtain that
$$\frac{\varphi''(x)}{\varphi'(x)}=\frac{\ki_{\alpha}''(x)}{\ki_{\alpha}'(x)}\geq \frac{\ki_{2}''(x)}{\ki_{2}'(x)}=\frac{\psi''(x)}{\psi'(x)}$$
for all $x>0$ and $\alpha\geq 2.$ In other words, the function $\varphi$ is convex with respect to $\psi$ on $(0,\infty)$ for $\alpha\geq 2.$

{\bf g.} Let us consider the following interpolation of the Gr\"uss inequality \cite{dragomir}: If the integrable functions
$f,g:[a,b]\to\mathbb{R}$ satisfies the inequalities $m_1\leq f(x)\leq M_1$ and $m_2\leq g(x)\leq M_2$ for all $x\in[a,b]$
and $p:[a,b]\to[0,\infty)$ is integrable such that $\int_a^bp(t)\dt>0,$ then
   \begin{align*}\left|\int_a^bp(t)\dt\right.&\left.\int_a^bp(t)f(t)g(t)\dt-\int_a^bp(t)f(t)\dt\int_a^bp(t)g(t)\dt\right|\\& \leq
           \frac{1}{4}(M_2-m_2)(M_1-m_1)\left[\int_a^bp(t)\dt\right]^2.\end{align*}
We use this inequality for the functions $f,g$ and $p$ as in the proof of part {\bf f}. Observe that when $\beta\leq0$ and $\alpha+\beta\geq0,$ then we have $0<f(t)<1$ and $0<g(t)<1$ for all $t>0.$

{\bf h.} \& {\bf i.} Owing to Kimberling \cite{kimberling} it is known that if the function $f,$ defined on $(0,\infty),$ is continuous
and completely monotonic and maps $(0,\infty)$ into $(0,1),$ then $\log f$ is super-additive, that is for all $x,y>0$ we have
   $$\log f(x+y)\geq \log f(x)+\log f(y)\ \ \mbox{or}\ \ f(x+y)\geq f(x)f(y).$$
In view of part {\bf a} the Bickley function $\ki$ is completely monotonic and so is $x\mapsto \ki_{\alpha}(x)/\ki_{\alpha}(0),$ which
maps  $(0,\infty)$ into $(0,1).$ Similarly, the function $\alpha\mapsto \ki_{\alpha}(x)/\ki_{0}(x)$ is completely monotonic on $(0,\infty),$
according to part {\bf b} of this theorem, and maps $(0,\infty)$ into $(0,1).$ Consequently, applying Kimberling's result, the proof of the left-hand side of the inequalities \eqref{ineq3} and \eqref{ineq4} is complete. Here we used that \cite[p. 259]{nist}
   \begin{equation}\label{alpha}
      \ki_{\alpha}(0) = \frac{\sqrt{\pi}\,\Gamma\left(\frac{\alpha}{2}\right)}{2\Gamma\left(\frac{\alpha +1}{2}\right)}
   \end{equation}
for all $\alpha>0.$

For the proof of the second inequalities in \eqref{ineq3} and \eqref{ineq4} recall the well-known fact that if for a function $g:(0,\infty)\to(0,\infty)$ we have that $x\mapsto g(x)/x$ is decreasing, then we have that $g$ is sub-additive, that is, for all $x,y>0$ one has $$g(x+y)\leq g(x)+g(y).$$ Now, both of functions $x\mapsto \ki_{\alpha}(x)/x$ and $\alpha\mapsto\ki_{\alpha}(x)/\alpha$ are decreasing on $(0,\infty),$ and hence the functions $x\mapsto \ki_{\alpha}(x)$ and $\alpha\mapsto\ki_{\alpha}(x)$ are sub-additive.

Now, we consider the proof of the last inequalities in \eqref{ineq3} and \eqref{ineq4}. In view of parts {\bf a} and {\bf b} the functions $x\mapsto \ki_{\alpha}'(x)$ and $\alpha\mapsto \partial \ki_{\alpha}(x)/\partial\alpha$ are increasing on $(0,\infty).$ Hence by using the monotone form of l'Hospital's rule \cite[Lemma 2.2]{vuorinen}, the functions $x\mapsto (\ki_{\alpha}(x)-\ki_{\alpha}(0))/x$ and $\alpha\mapsto (\ki_{\alpha}(x)-\ki_0(x))/\alpha$ are increasing too on $(0,\infty),$ which implies that the functions $x\mapsto \ki_{\alpha}(x)-\ki_{\alpha}(0)$ as well as $\alpha\mapsto \ki_{\alpha}(x)-\ki_0(x)$ are super-additive on $(0,\infty).$ Note that the last inequalities in \eqref{ineq3} and \eqref{ineq4} can proved also by using Petrovi\'c's result \cite[p. 22]{mitri}: if $f \colon [0,\infty) \mapsto \mathbb{R}$ is convex, then for all $x,y>0$ we have
   $$f(x)+f(y)\leq f(x+y)+f(0).$$

Finally, let us consider Vasi\'c's extension of Petrovi\'c inequality \cite{Vasic} which reads:
for a function $f$ convex on $[0,a]$, $x_j \in [0,a], p_j \geq 1, j\in\{1,2,\dots,n\}$; $\sum_{j=1}^n p_jx_j \in [0,a]$ there holds
   \[ \sum_{j=1}^n p_j f(x_j) \leq f\left( \sum_{j=1}^n p_jx_j\right) + \left( \sum_{j=1}^n p_j-1\right) f(0)\, .\]
Specifying $f=\ki_\alpha; n = 2, p_1=r, p_2=s; x_1 = x, x_2 = y$ and by the above exposed Vasi\'c's result we deduce \eqref{ineq31}. Let us point out that $r=s=1$ in \eqref{ineq31} gives the right-hand side inequality in \eqref{ineq3}.

{\bf j.} By the definition of exponential convexity and by using the left-hand side of \eqref{ineq3} we conclude
\begin{align*} H_\xi({\rm Ki}_\alpha) &=   \sum_{j, k = 1}^n \xi_j\xi_k {\rm Ki}_\alpha(x_j+x_k)\\
                             &\ge \frac{2\Gamma\left(\frac{\alpha+1}{2}\right)}{\sqrt{\pi}\,\Gamma\left(\frac{\alpha}{2}\right)}
                                 \sum_{j, k = 1}^n \xi_j\xi_k {\rm Ki}_\alpha(x_j){\rm Ki}_\alpha(x_k)\\
                             &=   \frac{2\Gamma\left(\frac{\alpha+1}{2}\right)}{\sqrt{\pi}\,\Gamma\left(\frac{\alpha}{2}\right)}
                                 \left[ \sum_{j= 1}^n \xi_j {\rm Ki}_\alpha(x_j)\right]^2>0\, , \end{align*}
where $n\in\{1,2,\dots\}$ and $\xi_j\in\mathbb{R},$ $x_j>0$ for $j\in\{0,1,\dots,n\}.$ Thus, $\ki_\alpha$ is exponentially convex on $(0,\infty)$ for $\alpha>0$. Now, because the exponential convexity implies log-convexity, we proved part {\bf d} for $\alpha>0$ as well.

{\bf k.} Similarly, by using the left-hand side of \eqref{ineq4} we conclude
\begin{align*} H_\eta({\rm Ki}_\alpha) &=   \sum_{j, k = 1}^n \eta_j\eta_k {\rm Ki}_{\alpha_j+\alpha_k}(x)\\
                             &\ge \frac{1}{{\rm Ki}_0(x)}
                                 \sum_{j, k = 1}^n \eta_j\eta_k {\rm Ki}_{\alpha_j}(x){\rm Ki}_{\alpha_k}(x)\\
                             &=  \frac{1}{{\rm Ki}_0(x)}
                                 \left[ \sum_{j= 1}^n \eta_j {\rm Ki}_{\alpha_j}(x)\right]^2>0\, , \end{align*}
where $n\in\{1,2,\dots\}$ and $\eta_j,\alpha_j\in\mathbb{R}$ for each $j\in\{0,1,\dots,n\}.$ Consequently, $\alpha\mapsto \ki_\alpha(x)$ is exponentially convex on $\mathbb{R}$ for $x>0$. Since the exponential convexity implies log-convexity, we proved part {\bf c} as well.

{\bf l.} \& {\bf m.} Employing the inequality $x + 1/x \ge 2$ we conclude \eqref{ineqX}. Indeed,
making use of the integral form \eqref{A0} of $\ki_\alpha(x)$ we have
   \begin{align*}
      \ki_{\alpha+\beta}(x) + \ki_{\alpha-\beta}(x) &= \int_0^{\infty}e^{-x\cosh t}(\cosh t)^{-\alpha}\left[(\cosh t)^\beta
                          + (\cosh t)^{-\beta}\right]\dt \\
                        & \ge 2 \int_0^{\infty}e^{-x\cosh t}(\cosh t)^{-\alpha}\dt = 2\, \ki_\alpha(x)\, .
   \end{align*}
Repeating this procedure to the left--hand side expression in \eqref{ineqXX} we get
   \begin{align*}
      \ki_{\alpha+\nu}(x) \cdot \ki_{\alpha-\mu}(x) &+ \ki_{\alpha-\nu}(x)\cdot \ki_{\alpha+\mu}(x) =
                  \int_0^\infty \int_0^\infty e^{-x(\cosh t+\cosh s)}(\cosh t\cosh s)^{-\alpha} \\
           &\times\, \left[(\cosh t)^\nu (\cosh s)^{-\mu} + (\cosh t)^{-\nu}(\cosh s)^\mu \right] \dt \ds \\
           &\ge 2 \int_0^\infty \int_0^\infty e^{-x(\cosh t+\cosh s)}(\cosh t\cosh s)^{-\alpha} \dt\ds
            = 2\left[\ki_\alpha(x)\right]^2\, ,
   \end{align*}
which finishes the proof of {\bf l.} It should be mentioned here that inequality \eqref{ineqX} is actually a consequence of part {\bf b} or {\bf c}. More precisely, since $\alpha \mapsto K_{\alpha}(x)$ is convex on $\mathbb{R}$ for all $x>0,$ we have
$$\ki_{\lambda\nu+(1-\lambda)\mu}(x)\leq \lambda\ki_{\nu}(x)+(1-\lambda)\ki_{\mu}(x)$$
for all $\nu,\mu\in\mathbb{R}$ and $x>0.$ Now, choosing $\nu=\alpha+\beta,$ $\mu=\alpha-\beta$ and $\lambda=1/2,$ we get \eqref{ineqX}.

{\bf n.} By using \eqref{A0} and the H\"older-Rogers inequality \eqref{holder} we obtain
\begin{align*}
\ki_{\lambda\alpha+(1-\lambda)\beta}(\lambda x+(1-\lambda)y)&=\int_0^{\infty}e^{-(\lambda x+(1-\lambda)y)\cosh t}(\cosh t)^{-(\lambda\alpha+(1-\lambda)\beta)}\dt\\
&=\int_0^{\infty}\left[e^{-x\cosh t}(\cosh t)^{-\alpha}\right]^{\lambda}\left[e^{-y\cosh t}(\cosh t)^{-\beta}\right]^{1-\lambda}\dt\\
&\leq\left[\int_0^{\infty}e^{-x\cosh t}(\cosh t)^{-\alpha}\dt\right]^{\lambda}\left[\int_0^{\infty}e^{-y\cosh t}(\cosh t)^{-\beta}\dt\right]^{1-\lambda}\\
&=\left[\ki_{\alpha}(x)\right]^{\lambda}\left[\ki_{\beta}(y)\right]^{1-\lambda}
\end{align*}
holds for all $\lambda\in[0,1],$ $\alpha,\beta\in\mathbb{R}$ and
$x,y>0,$ i.e. the function $(\alpha,x)\mapsto \ki_{\alpha}(x)$ is
log-convex. Now, by choosing in the above inequality $\lambda=1/2,$ and changing $\alpha$ to $(1+\mu)\alpha,$ $\beta$ to $(1-\mu)\alpha,$ $x$ to $(1+\nu)x$ and $y$ to $(1-\nu)x,$ we obtain the inequality \eqref{logc2}. We note that inequality \eqref{logc2} also follows from the generalized Cauchy-Schwarz inequality \cite{jamei1,jamei2}
$$\left[\int_a^bf(t)g(t)\dt\right]^2\leq\int_a^b\left[f(t)\right]^{1+\nu}\left[g(t)\right]^{1+\mu}\dt\int_a^b\left[f(t)\right]^{1-\nu}\left[g(t)\right]^{1-\mu}\dt,$$
where $\nu,\mu\in\mathbb{R}$ and $f,g:[a,b]\to\mathbb{R}$ are integrable functions such that the above integrals exist.
\end{proof}

The next theorem contains some other functional inequalities for the Bickley function.

\begin{theorem}
\begin{enumerate}
\item[\bf a.] For all $\alpha>1/4$ and $x>0$ the following inequality holds
   \begin{equation} \label{kibound}
      \ki_{\alpha}(x)\leq\dfrac{\sqrt{\pi}\, e^{-x}\, \Gamma(\alpha - \frac14)}{2\Gamma(\alpha)\, x^{1/4}}.
   \end{equation}
\item[\bf b.] The inequality
   \begin{equation}\label{partial}
      \frac{1}{2}\left[\ki_{\alpha+1}(x)-\ki_{\alpha-1}(x)\right]
         < \frac{\partial \ki_{\alpha}(x)}{\partial\alpha} <
           \frac{1}{2}\left[\ki_{\alpha+2}(x)-\ki_{\alpha}(x)\right]
   \end{equation}
is valid for all $\alpha\in\mathbb{R}$ and $x>0.$
\item[\bf c.] The inequality
   \begin{equation}\label{carl}
      \left[\ki_{\alpha}(x)\right]^4\leq \frac{\pi^2}{2}\ki_{2\alpha}(2x)\ki_{2\alpha-2}(2x)
   \end{equation}
holds for all $\alpha\in\mathbb{R}$ and $x>0.$
\item[\bf d.] If $\alpha>0$ and $x>0,$ then
   \begin{equation}\label{oruljneki}\ki_{\alpha}(x) \leq \frac{\sqrt{\pi}\alpha^{\alpha}\Gamma(\alpha)}{2(ex)^{\alpha}\Gamma\left(\alpha+\frac{1}{2}\right)}.\end{equation}
\item[\bf e.] Let $p, q$ be conjugated H\"older exponents, $1/p+1/q = 1,$ $\min\{p,q\}>1$. Then for all $\alpha>0$ and $x>0$
the inequality
   \begin{align} \label{bilateral}
      \left[\ki_\alpha(0) - x\ki_{\alpha-1}(0)\right]_+ \leq \ki_\alpha(x)
          &\leq \left[ K_0(xp)\right]^{1/p}\,
                \left[\ki_{\alpha q}(0)\right]^{1/q} \nonumber \\
          &\leq \dfrac{\sqrt{\pi}}{2^{1/q+1/(2p)}\,p^{1/(2p)}} \,
                \left[\frac{\Gamma\left(\frac{\alpha q}{2}\right)}{\Gamma\left(\frac{\alpha q+1}{2}\right)}\right]^{1/q}
                \cdot \dfrac{e^{-x}}{x^{1/(2p)}}\, .
   \end{align}
holds. Here $[A]_+ = \max\{A, 0\}$.
\end{enumerate}
\end{theorem}

\begin{proof}[\bf Proof]
{\bf a.} Let us recall the familiar formula for the gamma function \cite[p. 139]{nist}
   \[ \dfrac{\Gamma(\alpha)}{z^\alpha} = \int_0^\infty e^{-zy}y^{\alpha-1} \dy, \]
where $z,\alpha>0.$ Putting $z = \cosh t$ for a $t \in (0, \infty)$ we get
   \[ (\cosh t)^{-\alpha} = \dfrac1{\Gamma(\alpha)} \int_0^\infty e^{-y\cosh t}y^{\alpha-1} \dy\, .\]
This together with the integral form \eqref{A0} of $\ki_\alpha(x)$ yields the double integral representation
   \[ \ki_\alpha(x) = \int_0^\infty e^{-x\cosh t} (\cosh t)^{-\alpha} \dt
                    = \dfrac1{\Gamma(\alpha)} \int_0^\infty \int_0^\infty e^{-(x+y)\cosh t}\,y^{\alpha-1} \dt\dy\, .\]
Since the integrand is positive and $\cosh t \ge 1+t^2/2$ for all $t\in \mathbb R$, it follows that
   \[ \ki_\alpha(x) \le \dfrac1{\Gamma(\alpha)} \int_0^\infty\int_0^\infty e^{-(x+y)\left(1+{t^2}/2\right)}\,y^{\alpha-1} \dt\dy
                     =   \dfrac{e^{-x}}{\Gamma(\alpha)}\int_0^\infty e^{-y}y^{\alpha-1}
                         \left( \int_0^\infty e^{-\frac{x+y}2\, t^2}\, \dt\right)\, \dy =R.\]
The integration order exchange and the variable substitution $t\sqrt{(x+y)/2} \mapsto s$ lead to
   \[ R = \dfrac{\sqrt{2}\, e^{-x}}{\Gamma(\alpha)} \int_0^\infty \dfrac{e^{-y}y^{\alpha-1}}{\sqrt{x+y}}
                 \left(\int_0^\infty e^{-s^2}\, \ds\right) \dy
        = \sqrt{\dfrac\pi2}\,\dfrac{e^{-x}}{\Gamma(\alpha)}\, \int_0^\infty \dfrac{e^{-y}y^{\alpha-1}}{\sqrt{x+y}}\, \dy\,.\]
Applying the arithmetic mean - geometric mean inequality to the denominator of the integrand, we get
    \[ R \le \dfrac{\sqrt{\pi}\, e^{-x}}{2\Gamma(\alpha)\, x^{1/4}}\, \int_0^\infty e^{-y}y^{\alpha-1/4-1}\dy
                     =   \dfrac{\sqrt{\pi}\, e^{-x}\, \Gamma(\alpha - \frac14)}{2\Gamma(\alpha)\, x^{1/4}}\, ,\]
which makes sense for all $\alpha>\tfrac14$. This completes the proof of \eqref{kibound}.

{\bf b.} To prove \eqref{partial} observe that by \eqref{A0} we have
$$\frac{\partial \ki_{\alpha}(x)}{\partial\alpha}=-\int_0^{\infty}e^{-x\cosh t}\log(\cosh t)(\cosh t)^{-\alpha}\dt.$$
On the other hand, it is known \cite[Theorem 3.3]{neuman} that for all $t\in\mathbb{R}$ the inequality
$$2\left(\tanh \frac{t}{2}\right)^2<\log(\cosh t)<\frac{\sinh t\tanh t}{2}$$
is valid. Applying this inequality together with
$$\tanh \frac{t}{2}=\frac{\sinh t}{\cosh t+1}\geq\frac{\tanh t}{2},$$
where $t\in\mathbb{R},$ we obtain that for all $\alpha\in\mathbb{R}$ and $x>0$
\begin{align*}
\frac{\partial \ki_{\alpha}(x)}{\partial\alpha}&>-\frac{1}{2}\int_0^{\infty}e^{-x\cosh t}(\sinh^2t)(\cosh t)^{-(\alpha+1)}\dt\\
&=\frac{1}{2}\int_0^{\infty}e^{-x\cosh t}(1-\cosh^2t)(\cosh t)^{-(\alpha+1)}\dt\\
&=\frac{1}{2}\left(\int_0^{\infty}e^{-x\cosh t}(\cosh t)^{-(\alpha+1)}\dt-\int_0^{\infty}e^{-x\cosh t}(\cosh t)^{-(\alpha-1)}\dt\right)\\
&=\frac{1}{2}\left[\ki_{\alpha+1}(x)-\ki_{\alpha-1}(x)\right]
\end{align*}
and
\begin{align*}
\frac{\partial \ki_{\alpha}(x)}{\partial\alpha}&<-2\int_0^{\infty}e^{-x\cosh t}\left(\tanh \frac{t}{2}\right)^2(\cosh t)^{-(\alpha+1)}\dt\\
&\leq-\frac{1}{2}\int_0^{\infty}e^{-x\cosh t}(\tanh^2t)(\cosh t)^{-\alpha}\dt\\
&=\frac{1}{2}\int_0^{\infty}e^{-x\cosh t}(1-\cosh^2t)(\cosh t)^{-(\alpha+2)}\dt\\
&=\frac{1}{2}\left(\int_0^{\infty}e^{-x\cosh t}(\cosh t)^{-(\alpha+2)}\dt-\int_0^{\infty}e^{-x\cosh t}(\cosh t)^{-\alpha}\dt\right)\\
&=\frac{1}{2}\left[\ki_{\alpha+2}(x)-\ki_{\alpha}(x)\right].
\end{align*}

{\bf c.} We shall apply Carlson's inequality \cite[p. 370]{mitri} which states that if for the function $f:[0,\infty)\to(0,\infty)$ the functions
$x\mapsto \left[f(x)\right]^2$ and $x\mapsto \left[xf(x)\right]^2$ are integrable on $[0,\infty),$ then
$$\left[\int_{0}^{\infty}f(t)\dt\right]^4\leq \pi^2\left[\int_{0}^{\infty}[f(t)]^2\dt\right]\left[\int_{0}^{\infty}[tf(t)]^2\dt\right].$$
Thus we obtain that for all $\alpha\in\mathbb{R}$ and $x>0$
$$\left[\int_0^{\infty}e^{-x\cosh t}(\cosh t)^{-\alpha}\dt\right]^4\leq\pi^2\left[\int_0^{\infty}e^{-2x\cosh t}(\cosh t)^{-2\alpha}\dt\right]\left[\int_0^{\infty}t^2e^{-2x\cosh t}(\cosh t)^{-2\alpha}\dt\right].$$
Applying again the inequality $\cosh t\geq 1+t^2/2\geq\sqrt{2}t,$ we get $\cosh^2t\geq 2t^2,$ for all $t\geq 0.$ Consequently,
$$\int_0^{\infty}t^2e^{-2x\cosh t}(\cosh t)^{-2\alpha}\dt\leq\frac{1}{2}\int_0^{\infty}e^{-2x\cosh t}(\cosh t)^{-(2\alpha-2)}\dt.$$
Now, using the representation \eqref{A0}, we complete the proof of \eqref{carl}.

{\bf d.} Applying the inequality \cite[p. 266]{mitri} $e^{-y}\leq (a/e)^ay^{-a},$ where $y>0$ and $a>0,$ for $y=x\cosh t$ and $a=\alpha,$ we obtain
$$\ki_{\alpha}(x)=\int_0^{\infty}e^{-x\cosh t}(\cosh t)^{-\alpha}\dt\leq (\alpha/e)^{\alpha}x^{-\alpha}\int_{0}^{\infty}(\cosh t)^{-2\alpha}\dt=(\alpha/e)^{\alpha}x^{-\alpha}\ki_{2\alpha}(0).$$
Thus, in view of \eqref{alpha}, the inequality \eqref{oruljneki} follows.

{\bf e.} Applying the estimate $e^{-a} \geq 1-a$ with $a= x \cosh t$ in \eqref{A0} we obtain
   \[ \ki_\alpha(x) \geq \int_0^\infty (1-x \cosh t) (\cosh t)^{-\alpha} \dt = \ki_\alpha(0) - x\ki_{\alpha-1}(0).\]
This proves the left-hand side of \eqref{bilateral}. Now, let $p, q$ with $\min\{p,q\}>1$ be conjugated H\"older exponents. Then by the H\"older--Rogers inequality \eqref{holder} we conclude
   \begin{align*}
      \ki_\alpha(x) &= \int_0^\infty e^{-x \cosh t} \cdot (\cosh t)^{-\alpha} \dt \\
         &\le \left[ \int_0^\infty e^{-xp \cosh t} \dt\right]^{1/p}\, \left[ \int_0^\infty  (\cosh t)^{-\alpha q} \dt\right]^{1/q} \\
         &= \big[ \ki_0(xp)\big]^{1/p} \cdot \big[ \ki_{\alpha q}(0)\big]^{1/q}
          = \big[ K_0(xp)\big]^{1/p}\, \left[\frac{\sqrt{\pi}\,\Gamma\left(\frac{\alpha q}{2}\right)}
            {2\Gamma\left(\frac{\alpha q+1}{2}\right)}\right]^{1/q} \, .
    \end{align*}
On the other hand, since $\cosh t\geq 1+t^2/2,$
    \[ \ki_0(xp) \le \int_0^\infty e^{-xp(1+t^2/2)} \dt = \dfrac{\sqrt{\pi}\, e^{-xp}}{\sqrt{2xp}}\, ,\]
which completes the proof of Theorem 2.
\end{proof}

Finally, observe that the Tur\'an type inequality \eqref{turanka} can be deduced also from the representation
$$\ki_{\alpha-2}(x)\ki_{\alpha}(x)-\left[\ki_{\alpha-1}(x)\right]^2=\frac{1}{2}\int_0^{\infty}\int_0^{\infty}e^{-x(\cosh t+\cosh s)}(\cosh t\cosh s)^{-\alpha}(\cosh t-\cosh s)^2\dt\ds.$$
Moreover, the above integral representation yields the following complete monotonicity result: the function
$$x\mapsto\left|\begin{array}{cc}\ki_{\alpha}(x)&\ki_{\alpha-1}(x)\\\ki_{\alpha-1}(x)&\ki_{\alpha-2}(x)\end{array}\right|$$
is not only positive, but also completely monotonic on $(0,\infty)$ for all $\alpha\in\mathbb{R}.$ The next result generalizes this property of the Bickley function concerning Tur\'an determinants.

\begin{theorem}
If $\alpha\in\mathbb{R}$ and $n\in\{1,2,\dots\},$ then the function
$$x\mapsto \mathrm{Det}_{\ki}(x)=\left|\begin{array}{cccc}\ki_{\alpha}(x)&\ki_{\alpha-1}(x)&\cdots&\ki_{\alpha-n}(x)\\
\ki_{\alpha-1}(x)&\ki_{\alpha-2}(x)&\cdots&\ki_{\alpha-n-1}(x)\\\vdots&\vdots&&\vdots\\
\ki_{\alpha-n}(x)&\ki_{\alpha-n-1}(x)&\cdots&\ki_{\alpha-2n}(x)
\end{array}\right|$$
is completely monotonic on $(0,\infty).$
\end{theorem}

\begin{proof}[\bf Proof]
Recently, Baricz and Ismail \cite[Theorem 5]{tricomi} proved the following result (see also \cite[Remark 2.9]{ismail}): If the sequence of functions $\{f_n\}_{n\geq0}$ is of the form
$$f_n(x) = \int_{\alpha}^{\beta} [\f(t,x)]^n d\mu(t,x),$$ where
$\f,\mu:[\alpha,\beta]\times\mathbb{R}\to\mathbb{R}$ and $\alpha,\beta\in\mathbb{R}$ such that $\alpha<\beta,$ then the determinant
$$\mathrm{Det}_n(x)=\left| \begin{array}{cccc}
f_0(x) & f_{1}(x) & \cdots & f_{n}(x)\\
f_{1}(x) & f_{2}(x) & \cdots & f_{n+1}(x)\\
\vdots & \vdots &  & \vdots\\
f_{n}(x) & f_{n+1}(x) & \cdots & f_{2n}(x)\\
\end{array} \right|$$
can be rewritten as follows
$$
\mathrm{Det}_n(x) = \frac{1}{(n+1)!} \int_{[\alpha,\beta]^{n+1}}
\prod_{0 \le j < k \le n} [\f(t_j,x) - \f(t_k,x)]^2 \prod_{j=0}^n
d\mu(t_j,x).
$$
Applying this result for the sequence of functions $\{\ki_{\alpha-n}\}_{n\geq0}$ we obtain that
$$\mathrm{Det}_{\ki}(x)=\frac{1}{(n+1)!}\int_{(0,\infty)^{n+1}}\exp\left(-x\sum_{j=0}^n\cosh t_j\right)\prod_{0\leq j<k\leq n}(\cosh t_j-\cosh t_k)^2\prod_{j=0}^n(\cosh t_j)^{-\alpha}dt_j,$$
which shows that indeed the function $x\mapsto \mathrm{Det}_{\ki}(x)$ is completely monotonic on $(0,\infty)$ for all $x>0.$
\end{proof}

\subsection*{Acknowledgement} The authors are very grateful to the eagle-eyed referee for his/her several constructive comments and suggestions, which improved the quality of the paper.

\end{document}